\documentclass[12pt,leqno]{article}
\usepackage{amsfonts,amsthm,amsmath}
\newtheorem{thm}{Theorem}
\newtheorem{prop}[thm]{Proposition}
\newtheorem{lem}[thm]{Lemma}
\newtheorem{cor}[thm]{Corollary}
\theoremstyle{remark}

\newcommand{\FF}{\mathbb{F}}
\newcommand{\ZZ}{\mathbb{Z}}
\newcommand{\RR}{\mathbb{R}}

\newcommand{\1}{\mathbf{1}}

\newcommand{\cF}{\mathcal{F}}
\DeclareMathOperator{\Aut}{Aut}

\DeclareMathOperator{\kernel}{Ker}

\begin{document}
\title{A Complete Classification of Ternary Self-Dual Codes of Length 24}

\author{
Masaaki Harada\thanks{
Department of Mathematical Sciences,
Yamagata University,
Yamagata 990--8560, Japan}
and 
Akihiro Munemasa\thanks{
Graduate School of Information Sciences,
Tohoku University,
Sendai 980--8579, Japan}
}

\maketitle

\begin{abstract}
Ternary self-dual codes have been classified for lengths
up to $20$.
At length $24$, a classification of only extremal self-dual codes 
is known.
In this paper, we give a complete classification
of ternary self-dual codes of length $24$
using the classification of $24$-dimensional odd unimodular lattices.
\end{abstract}

\section{Introduction}
\label{Sec:1}
As described in~\cite{RS-Handbook},
self-dual codes are an important class of linear codes for both
theoretical and practical reasons.
It is a fundamental problem to classify self-dual codes
of modest length 
and determine the largest minimum weight among self-dual codes
of that length.
By the Gleason--Pierce theorem, there are
nontrivial divisible self-dual
codes over $\FF_q$ for $q=2,3$ and $4$ only, where $\FF_q$  
denotes the
finite field of order $q$, and this is one of the reasons why
much work has been done concerning self-dual codes over 
these fields.

A code $C$ over $\FF_3$ is called ternary.
All codes in this paper are ternary.
A code $C$ of length $n$ is said to be {\em self-dual} if
$C=C^\perp$, where
the dual code $C^\perp$ of $C$ is defined as 
$C^\perp=\{x \in \FF_3^{n} |\ x \cdot y=0 \text{ for all } y\in C\}$
under the standard inner product $x \cdot y$. 
A self-dual code of length $n$ exists if and only if 
$n \equiv 0 \pmod 4$.
It was shown in~\cite{Mallows-Sloane} that the minimum weight 
$d$ of a self-dual code of length $n$ is bounded by $d\leq 3[n/12]+3$.
If $d=3[n/12]+3$, then the code is called {\em extremal}.  
Two codes $C$ and $C'$ are {\em equivalent} 
if there exists a monomial matrix $P$ with 
$C' = C \cdot P = \{ x P\:|\: x \in C\}$.  
The {\em automorphism group} $\Aut(C)$ of $C$ is the group of all
monomial matrices $P$ with
$C = C \cdot P$.

All self-dual codes of length $\leq 20$ have been 
classified~\cite{CPS,MPS,PSW}.
At length 24, the complete classification has not been done yet, 
although it was shown by Leon, Pless and Sloane \cite{LPS} 
that there are exactly two inequivalent extremal
self-dual codes, namely, the extended quadratic residue code $QR_{24}$
and the Pless symmetry code $P_{24}$.
Moreover, it was shown in \cite{LPS} that
there are at least $13$ and $96$ inequivalent self-dual codes
with minimum weights $3$ and $6$, respectively.

Applying the classification method in~\cite{HMV} to length $24$,
we give a classification of self-dual codes of length $24$
with minimum weights $3$ and $6$,
which completes the classification of self-dual
codes of length $24$.

\begin{thm}\label{thm}
There are exactly $166$ inequivalent ternary self-dual 
$[24,12,6]$ codes.
There are exactly $170$ inequivalent ternary self-dual 
$[24,12,3]$ codes.
\end{thm}

Generator matrices of all self-dual codes of length $24$
can be obtained electronically from~\cite{Data}.
All computer calculations in this paper
were done by {\sc Magma}~\cite{Magma}.

\section{Preliminaries}
\label{Sec:Pre}

An $n$-dimensional (Euclidean) lattice $L$ is {\em integral} if
$L \subseteq L^{*}$, where
the dual lattice $L^{*}$ is defined as 
$L^{*} = \{ x \in {\RR}^n | (x,y) \in \ZZ \text{ for all }
y \in L\}$ under the standard inner product $(x,y)$.
A lattice $L$ with $L=L^{*}$
is called {\em unimodular}.
The {\em norm} of a vector $x$ is $(x, x)$.
The {\em minimum norm} of $L$ is the smallest 
norm among all nonzero vectors of $L$.
A unimodular lattice $L$ is {\em even} if all vectors of $L$ have even norms,
and {\em odd} if some vector has an odd norm.
The {\em kissing number} of $L$ is the number of vectors of $L$ with
minimum norm.
Two lattices $L$ and $L'$ are {\em isomorphic}, denoted $L \cong L'$,
if there exists an orthogonal matrix $A$ with
$L' = L \cdot A$.
The {\em automorphism group} $\Aut(L)$ of $L$ is the group of all
orthogonal matrices $A$ with $L = L \cdot A$.

If $C$ is a self-dual code of length $n$ and minimum weight $d$,
then 
\[
A_{3}(C) = \frac{1}{\sqrt{3}} 
\{(x_1,\ldots,x_n) \in \ZZ^n \:|\: 
(x_1 \bmod 3,\ldots,x_n \bmod 3)\in C\}
\]
is a unimodular lattice with minimum norm $\min\{3,d/3\}$.
This construction of lattices from codes
is called Construction A\@. 

\begin{lem}\label{Lem:A3A6}
Let $C$ be a ternary self-dual code of length $n$.
Let $\alpha_i$ (resp.\ $\beta_{3i}$) be the number of vectors of
norm $i$ in $A_3(C)$ (resp.\ codewords of weight $3i$ in $C$)
($i=1,2$).
Then
\[
\alpha_2= \beta_6 + 3 \beta_3 \text{ and } \alpha_1= \beta_3.
\]
\end{lem}
\begin{proof}
The straightforward proof is omitted.
\end{proof}

The weight distribution of a self-dual code $C$
of length $24$ is determined by the numbers $\beta_3,\beta_6$
(see \cite[Table III]{LPS}).
Hence the weight distribution of $C$ 
can be determined by the numbers $\alpha_1,\alpha_2$ in $A_3(C)$.

There are $155$ non-isomorphic $24$-dimensional odd
unimodular lattices with minimum norm $2$, and
there are $117$ non-isomorphic $24$-dimensional odd
unimodular lattices with minimum norm $1$
\cite{Bor} (see also \cite[Table 2.2]{SPLAG}).
We denote the $i$-th $24$-dimensional odd
unimodular lattice with minimum norm $\ge2$
in~\cite[Table 17.1]{SPLAG} by $L_{24,i}$
($i=1,2,\ldots,156$).
The lattices $L_{24,i}$
($i=2,3,\ldots,156$) are the $155$
non-isomorphic $24$-dimensional odd
unimodular lattices with minimum norm $2$.
A $24$-dimensional unimodular lattice with minimum
norm $1$ except $\ZZ^{24}$ can be constructed as
$M_i \oplus \ZZ^{24-i}$ where $M_i$ is an
$i$-dimensional unimodular lattice with minimum norm $\ge 2$.
Here we denote a $24$-dimensional unimodular lattice 
$M_{i,j}\oplus \ZZ^{24-i}$ with minimum norm $1$ 
by $L_{i,j}$ where $M_{i,j}$ is the $j$-th 
$i$-dimensional unimodular lattice with minimum norm $\ge 2$ 
in~\cite[Table 16.7]{SPLAG}.
All non-isomorphic unimodular lattices with minimum norm $\ge 2$ 
can be constructed as neighbors of the standard lattices
for dimensions up to $24$
(see \cite[Tables I, II and III]{Ba97}).

A set $\{f_1, \ldots, f_{n}\}$ of $n$ vectors $f_1, \ldots, f_{n}$ in an
$n$-dimensional lattice $L$ with
$(f_i, f_j) = 3 \delta_{ij}$
is called a {\em $3$-frame} of $L$,
where $\delta_{ij}$ is the Kronecker delta.
Clearly, $A_3(C)$ has a $3$-frame.
Conversely, every self-dual code
can be obtained from a $3$-frame of some unimodular lattice.
Let $\cF=\{f_1,\dots,f_n\}$ be a $3$-frame of $L$. Consider the
mapping
\begin{align*}
&\pi_{\cF}:\frac13\oplus_{i=1}^n\ZZ f_i\to\FF_3^n\\
&\pi_{\cF}(x)=((x,f_i) \bmod 3)_{1\leq i\leq n}.
\end{align*}
Then $\kernel\pi_{\cF}=\oplus_{i=1}^n\ZZ f_i\subset L$, so
the code $C=\pi_{\cF}(L)$ satisfies $\pi_{\cF}^{-1}(C)=L$.
This implies $A_3(C)\cong L$ and every code $C$ with
$A_3(C)\cong L$ is obtained as $\pi_{\cF}(L)$ for some
$3$-frame $\cF$ of $L$.

\begin{lem}[\cite{HMV}]
\label{Lem:Key}
Let $L$ be an $n$-dimensional integral lattice, and let
$\cF=\{f_1,\dots,f_n\}$, $\cF'=\{f'_1,\dots,f'_n\}$ be
$3$-frames of $L$. Then the codes $\pi_{\cF}(L)$ and
$\pi_{\cF'}(L)$ are equivalent if and only if there exists an
orthogonal matrix $P\in\Aut(L)$ such that
$\{\pm f_1,\dots,\pm f_n\}\cdot P=
\{\pm f'_1,\dots,\pm f'_n\}$.
\end{lem}

In order to establish 
the nonexistence of a $3$-frame for some lattices, 
the shadows of lattices are considered.
Let $L=L_0\cup L_2$ be an odd unimodular lattice with 
even sublattice $L_0$. 
Then $L_0^*$ can be written as a union of cosets of $L_0$:
$L_0^*=L_0\cup L_2 \cup L_1 \cup L_3$.
The shadow $S$ of $L$ is defined to be $S=L_1\cup L_3$~\cite{CS90}.

\begin{lem}
Let $L=L_0\cup L_2$ be a $24$-dimensional odd unimodular lattice
with shadow $S=L_1\cup L_3$.
Let $v$ be a vector of $L_2$ with $(v,v)=3$.
If there exist vectors $a\in L_1$ and $b\in L_3$ such that
$(a,a)=(b,b)=2$, $(a,b)=1/2$ and $v=a-b$, then 
$v$ does not belong to any $3$-frame of $L$.
\end{lem}
\begin{proof}
Suppose that $\{f_1,\dots,f_{24}\}$ is a
$3$-frame of $L$ and $v=f_1$.
By Lemma~2 in~\cite{HKO}, a vector $a \in L_1$ can be written as
\[
a=\frac16\sum_{i=1}^{24} a_if_i,\quad
(a_i \in 1+2\ZZ).
\]
Since
\[
2=(a,a)=\frac{1}{12}\sum_{i=1}^{24}a_i^2
\ge \frac{1}{12}\sum_{i=1}^{24}1 =2,
\]
we have $a_i=\pm1$ for all $1\leq i\leq24$.
Then
\[
\frac32 
=(a,a)-(a,b)
=(a,v)
=\frac{a_1}{2}
=\pm\frac{1}{2}.
\]
This is a contradiction.
\end{proof}

Let $L$ be a $24$-dimensional unimodular lattice and
let $V$ be the set of pairs
$\{v,-v\}$ with $(v,v)=3$, $v \in L$ satisfying the
condition that
there do not exist $a\in L_1$ and $b\in L_3$ such that
$(a,a)=(b,b)=2$, $(a,b)=1/2$ and $v=a-b$.
We define the simple undirected graph $\Gamma$, whose set of  
vertices
is the set $V$ and
two vertices $\{v,-v\},\{w,-w\}\in V$  
are adjacent if $(v,w)=0$.
It follows that the $3$-frames are precisely the $24$-cliques
in the graph $\Gamma$.
It is clear that $\Aut(L)$ acts on the graph $\Gamma$
as automorphisms, and
Lemma~\ref{Lem:Key} implies that the $\Aut(L)$-orbits 
on the set of $24$-cliques of $\Gamma$ are in one-to-one
correspondence with the equivalence classes of codes
$C$ satisfying $A_3(C)\cong L$.
Therefore, the classification of such codes reduces to
finding a set of representatives of
$24$-cliques of $\Gamma$ up to the action of $\Aut(L)$. 
This computation was performed by {\sc Magma}~\cite{Magma}, 
the results were then converted to $3$-frames, and then
to self-dual codes of length $24$.
In this way, by considering $3$-frames of all $24$-dimensional 
odd unimodular lattices with minimum norms $2$ and $1$,
we have all inequivalent self-dual codes of length $24$
with minimum weights $6$ and $3$, respectively.

Note that the graph $\Gamma$ is an empty graph
for the lattice $L_{i,j}$, where
\begin{align*}
(i,j)=
&
(20,  1),(22,  1),(23,  2),(23,  3),(24, 94),(24,125),(24,126),
\\ &
(24,135),(24,136),(24,137),(24,143),(24,147),(24,148),
\\ &
(24,149),(24,151),(24,152),(24,153),(24,155),(24,156).
\end{align*}
In particular,
none of these lattices have a $3$-frame.

\section{Decomposable self-dual codes of length 24}
\label{Sec:Dec}
As described in \cite[Table~II]{LPS}, there are $27$ inequivalent
decomposable self-dual codes $D_i$ $(i=1,2,\ldots,27)$
of length $24$.
A decomposable self-dual code can be written as
$C_1 \oplus C_2$ where $C_1$ and $C_2$ are 
self-dual codes of lengths $20$ and $4$, or
both $C_1$ and $C_2$ are indecomposable self-dual codes of length $12$.
We denote the unique self-dual $[4,2,3]$ code 
in~\cite[Table~1]{MPS} by $E_4$.
We also denote 
the extended ternary Golay $[12,6,6]$ code
by $G_{12}$ and the unique indecomposable
$[12,6,3]$ code by $4C_3(12)$~\cite[Table~1]{MPS}.
We denote the self-dual codes of length $20$ by 
$C_{20,1},\ldots,C_{20,24}$ according to the order given in
\cite[Tables II and III]{PSW}.
In Table~\ref{Table:dec},
we list the number $\beta_3$ of the codewords of weight $3$
and the order $\#\Aut(D_i)$ of the automorphism group
for $D_i$ $(i=1,2,\ldots,27)$.

\begin{table}[htb]
\caption{Decomposable self-dual codes of length 24}
\label{Table:dec}
\begin{center}
{\scriptsize
\begin{tabular}{c|c|c|r||c|c|c|r}
\noalign{\hrule height0.8pt}
$i$ & $(C_1,C_2)$ & $\beta_3$ & \multicolumn{1}{c||}{$\#\Aut(D_i)$} &
$i$ & $(C_1,C_2)$ & $\beta_3$ & \multicolumn{1}{c}{$\#\Aut(D_i)$} \\
\hline
${ 1}$ & $(C_{20, 1}, E_4)$ & 48 & 8806025134080 &
${15}$ & $(C_{20,15}, E_4)$ & 12 &       1327104 \\
${ 2}$ & $(C_{20, 2}, E_4)$ & 32 &   41278242816 &
${16}$ & $(C_{20,16}, E_4)$ & 12 &       1161216 \\
${ 3}$ & $(C_{20, 3}, E_4)$ & 24 &  126127964160 &
${17}$ & $(C_{20,17}, E_4)$ & 10 &        110592 \\
${ 4}$ & $(C_{20, 4}, E_4)$ & 24 &    1146617856 &
${18}$ & $(C_{20,18}, E_4)$ & 10 &        331776 \\
${ 5}$ & $(C_{20, 5}, E_4)$ & 20 &     477757440 &
${19}$ & $(C_{20,19}, E_4)$ &  8 &        184320 \\
${ 6}$ & $(C_{20, 6}, E_4)$ & 18 &     310542336 &
${20}$ & $(C_{20,20}, E_4)$ &  8 &         24576 \\
${ 7}$ & $(C_{20, 7}, E_4)$ & 16 &     198180864 &
${21}$ & $(C_{20,21}, E_4)$ &  8 &        491520 \\
${ 8}$ & $(C_{20, 8}, E_4)$ & 20 &     644972544 &
${22}$ & $(C_{20,22}, E_4)$ &  8 &         92160 \\
${ 9}$ & $(C_{20, 9}, E_4)$ & 18 &      89579520 &
${23}$ & $(C_{20,23}, E_4)$ &  8 &        497664 \\
${10}$ & $(C_{20,10}, E_4)$ & 16 &      15925248 &
${24}$ & $(C_{20,24}, E_4)$ &  8 &      99532800 \\
${11}$ & $(C_{20,11}, E_4)$ & 14 &     985374720 &
${25}$ & $(G_{12}, G_{12})$  &  0 &   72260812800\\
${12}$ & $(C_{20,12}, E_4)$ & 14 &       2985984 &
${26}$ & $(4C_3(12), 4C_3(12))$& 16 &  7739670528\\
${13}$ & $(C_{20,13}, E_4)$ & 12 &      19906560 &
${27}$ & $(4C_3(12), G_{12})$&  8 &   11824496640\\
${14}$ & $(C_{20,14}, E_4)$ & 12 &       2985984 &
       &                    &    &  \\ 
\noalign{\hrule height0.8pt}
   \end{tabular}
}
\end{center}
\end{table}

There is a unique decomposable self-dual $[24,12,6]$ code. 
This can also be established from the classification of 
odd unimodular lattices, by the following lemma.

\begin{lem}\label{Lem:Dec}
Let $C$ be a ternary self-dual code of length $n$ and
minimum weight at least $6$. Then $C$ is decomposable if
and only if $A_3(C)$ is decomposable.
\end{lem}
\begin{proof}
If $C$ is decomposable, then obviously $A_3(C)$ is decomposable.
Conversely, suppose 
that $A_3(C)=L \oplus L'$ for some sublattices $L,L'$.
Since $A_3(C)$ has minimum norm $\ge 2$,
both $L$ and $L'$ have minimum norms $\ge $2.
Let $x$ be a vector of norm 3 in $A_3(C)$.
Then $x$ can be written as
either $(x_1,0)$ or $(0,x_2)$ where $x_1 \in L$ and $x_2 \in L'$.
Hence, every $3$-frame of $A_3(C)$ is a union of those of $L$
and of $L'$.
Therefore, $C$ is decomposable.
\end{proof}

If $C$ is a decomposable self-dual $[24,12,6]$ code, then
Lemma~\ref{Lem:Dec} implies that
$A_3(C)$ is a decomposable odd unimodular lattice with
minimum norm $2$. The only such
lattices are $L_{24,153}=E_8 \oplus (D_8 \oplus D_8)^+$
and $L_{24,154}=D_{12}^+ \oplus D_{12}^+$.
Clearly, $L_{24,153}$ has no $3$-frame, since the lattice
$E_8$ does not have one.
Since the extended ternary Golay $[12,6,6]$ code $G_{12}$
is the unique code $C$ with $A_3(C) \simeq D_{12}^+$,
the decomposable code $D_{25}$ is the unique code $C$
with $A_3(C) \simeq L_{24,154}$, up to equivalence.

\section{Self-dual [24, 12, 6] codes}
\label{Sec:d6}
In this section, we give a classification of 
self-dual $[24,12,6]$ codes by considering
$3$-frames of $24$-dimensional odd unimodular lattices
with minimum norm $2$.

By the approach described in Section~\ref{Sec:Pre},
we completed the classification of 
self-dual $[24,12,6]$ codes.
In Table~\ref{Table:d6}, we list the number $N_i$ 
of inequivalent self-dual $[24,12,6]$ codes $C$
with $A_3(C) \cong L_{24,i}$.
The columns $\#\Aut$ in the table list the
orders of automorphism groups.
From Table~\ref{Table:dec},
the only decomposable self-dual $[24,12,6]$ code is $D_{25}=
G_{12}\oplus G_{12}$ where the automorphism group order is
marked in Table~\ref{Table:d6}.
We remark that there is no self-dual $[24,12,6]$ code $C$
with $A_3(C) \cong L_{24,i}$ unless $i$ is listed
in Table~\ref{Table:d6}.

\begin{table}[htbp]
\caption{Ternary self-dual $[24,12,6]$ codes}
\label{Table:d6}
\begin{center}
{\footnotesize
\begin{tabular}{c|r|l||c|r|l}
\noalign{\hrule height0.8pt}
$i$  &  \multicolumn{1}{c|}{$N_i$} 
& \multicolumn{1}{c||}{$\#\Aut$} &
$i$  &  \multicolumn{1}{c|}{$N_i$} 
& \multicolumn{1}{c}{$\#\Aut$} \\
\hline
  2  &  2  & 64, 64                                      &
 32  &  2  & 3456$^{C(H7)}$, 31104$^{C(H2)}$                   \\
  3  &  5  & 8, 12, 24, 24, 3072                         &
 33  &  2  & 48, 576                                     \\
  4  &  4  & 16, 16, 32, 512                             &
 34  &  1  & 192                                         \\
  5  & 10  & 4, 4, 8, 8, 12, 16, 24, 32, 32, 72          &
 35  &  1  & 64                                          \\
  6  &  2  & 960$^{C(H4)}$, 3072$^{C(H6)}$                     &
 36  &  1  & 512                                         \\
  7  & 11  & 2$^{Tri}$, 4, 4, 4, 4, 4, 8, 8, 16, 16, 32          &
 37  &  1  & 256                                         \\
  8  &  8  & 4, 8, 12, 16, 16, 24, 32, 64                &
 40  &  1  & 256                                         \\
  9  &  8  & 16, 24, 48, 64, 64, 96, 192, 384            &
 41  &  1  & 256                                         \\
 10  & 12  & 4, 4, 4, 8, 12, 16, 16, 16, 24, 24, 48, 384 &
 43  &  1  & 5184                                        \\
 11  &  5  & 12, 16, 48, 128, 1728                       &
 44  &  2  & 384, 3456                                   \\
 12  &  4  & 16, 48, 48, 1728                            &
 46  &  1  & 512                                         \\
 13  & 11  & 4, 4, 4, 8, 8, 8, 16, 16, 16, 16, 16        &
 47  &  1  & 1728                                        \\
 14  &  1  & 64                                          &
 50  &  1  & 8192                                        \\
 15  &  4  & 16, 32, 32, 2048                            &
 53  &  1  & 1024                                        \\
 16  &  4  & 8, 8, 16, 32                                &
 59  &  1  & 2304                                        \\
 17  &  7  & 8, 32, 32, 32, 32, 96, 1152                 &
 61  &  1  & 17280                                       \\
 18  &  4  & 16, 32, 64, 384                             &
 62  &  1  & 5184                                        \\
 19  &  6  & 32, 48, 48, 192, 648, 864                   &
 63  &  1  & 1728                                        \\
 20  &  1  & 384                                         &
 64  &  1  & 34560                                       \\
 21  &  2  & 16, 64                                      &
 65  &  1  & 3456                                        \\
 22  &  3  & 8, 16, 128                                  &
 73  &  1  & 20736                                       \\
 23  &  2  & 64, 128                                     &
 74  &  1  & 49152$^{C(H3)}$                                \\
 24  &  4  & 32, 64, 64, 256                             &
 78  &  2  & 20736, 241920$^{g_{10}+\eta_{14}}$       \\
 25  &  2  & 2304, 3072                                  &
 87  &  1  & 276480                                      \\
 26  &  1  & 288                                         &
102  &  1  & 746496                                      \\
 27  &  2  & 288, 576                                    &
114  &  1  & 622080                                      \\
 28  &  1  & 96                                          &
115  &  1  & 746496                                      \\
 29  &  2  & 128, 768                                    &
130  &  1  & 8294400$^{C(H5)}$                              \\
 30  &  2  & 256, 768                                    &
141  &  1  & 88957440$^{g_{11}+p_{13}}$             \\
 31  &  2  & 96, 1296                                    &
154  &  1  & 72260812800$^{D_{25},C(H1)}$                       \\
\noalign{\hrule height0.8pt}
   \end{tabular}
}
\end{center}
\end{table}

By Lemma~\ref{Lem:A3A6},
the weight enumerator of a self-dual $[24,12,6]$ code $C$
is determined by the kissing number of the lattice $A_3(C)$.
Since $A_3(C)$ is isomorphic to one of the $155$ lattices whose
kissing numbers are given in~\cite[Table III]{Ba97}, we do not
give the weight enumerator of $C$ or the number of codewords
of weight $6$ in $C$ for codes $C$ given in Table~\ref{Table:d6}.

Let $\mathcal{C}_i$ denote the set of all inequivalent
indecomposable self-dual codes of length $24$
containing exactly $2i$ codewords of weight $3$.
The following values
\begin{equation}\label{Eq:Ti}
T_i = \sum_{C \in \mathcal{C}_i} \frac{1}{\# \Aut(C)}
\quad(i=0,1,\dots,8)
\end{equation}
were determined theoretically in~\cite[Table I]{LPS}, without
finding the set $\mathcal{C}_i$. We now have the set
$\mathcal{C}_0$ as
\[
\mathcal{C}_0=\{QR_{24},P_{24}\} \cup \mathcal{C}
\setminus \{D_{25}\},
\]
where $\mathcal{C}$ is the set of the $166$ codes
given in Table~\ref{Table:d6}, and
we verified that the value $T_0$ obtained from our classification
coincided with the value determined in~\cite[Table I]{LPS}.
This shows that there is no other self-dual code
with minimum weight $d \ge 6$.

We investigate the previously known self-dual $[24,12,6]$ codes.
In~\cite{LTP} and \cite{LPS}, self-dual codes 
generated by the rows of Hadamard matrices
of order $24$ were studied.
The $60$ inequivalent Hadamard matrices of order $24$ give
exactly two inequivalent extremal self-dual codes~\cite{LPS}
and exactly seven inequivalent self-dual $[24,12,6]$ 
codes $C(H1),\ldots,C(H7)$ which are generated by the matrices
$H1,\ldots,H7$, respectively~\cite{LTP}.
We verified that all of these seven codes appear in the
present classification, and in Table~\ref{Table:d6} we mark 
the orders of automorphism groups for these codes.

Some properties of automorphism groups of 
self-dual $[24,12,6]$ codes were given in~\cite{LPS}.
For example, only the primes $2$, $3$, $5$, $7$, $11$
and $13$ can divide the orders of the automorphism groups.
A self-dual $[24,12,6]$ code with a trivial
automorphism group was given in~\cite[Fig.~2]{LPS}
and the authors conjectured that the code is unique.
Our classification shows that the conjecture is true.
We mark the code by $Tri$ in Table~\ref{Table:d6}. 
The code $g_{10}+\eta_{14}$ given in~\cite[Fig.~4]{LPS}
is the unique self-dual $[24,12,6]$ code with an automorphism 
of order $7$ and 
the code $g_{11}+p_{13}$ given in~\cite[Fig.~5]{LPS}
is the unique self-dual $[24,12,6]$ code with an automorphism 
of order $13$~\cite{LPS}.
In Table~\ref{Table:d6}
we mark the orders of automorphism groups for these two codes.
From our classification, only $g_{11}+p_{13}$ and 
the decomposable code $D_{25}$
are self-dual $[24,12,6]$ codes with an automorphism 
of order $11$.

\section{Self-dual [24, 12, 3] codes}
\label{Sec:d3}
In this section, we give a classification of 
self-dual $[24,12,3]$ codes by considering
$3$-frames in $24$-dimensional odd unimodular lattices
with minimum norm $1$.
These are the lattices $\ZZ^{24}$ and $L_{i,j}$ ($i \le 23$).

In Table~\ref{Table:d3}, we list the number $N$ 
of inequivalent self-dual $[24,12,3]$ codes $C$
with $A_3(C) \cong L$ for $L=\ZZ^{24}, L_{i,j}$.
The column $\#\Aut$ in Table~\ref{Table:d3} lists the
orders of automorphism groups.
In Table~\ref{Table:d3},
we mark the orders of automorphism groups for decomposable codes.
We remark that there is no self-dual $[24,12,3]$ code $C$
with $A_3(C) \cong L_{i,j}$ unless the lattice $L_{i,j}$ is
listed in Table~\ref{Table:d3}.

\begin{table}[htbp]
\caption{Ternary self-dual $[24,12,3]$ codes}
\label{Table:d3}
\begin{center}
{\footnotesize
\begin{tabular}{c|r|l}
\noalign{\hrule height0.8pt}
$L$&\multicolumn{1}{c|}{$N$}&\multicolumn{1}{c}{$\#\Aut$} \\
\hline
$\ZZ^{24}$   & 1 & 8806025134080$^{D_{1}}$                \\
\hline
$L_{8}$        & 1 & 41278242816$^{D_{2}}$                \\
\hline
$L_{12}$     & 2 & 1146617856$^{D_{4}}$, 126127964160$^{D_{3}}$    \\
\hline
$L_{14}$     & 2 & 477757440$^{D_{5}}$, 644972544$^{D_{8}}$    \\
\hline
$L_{15}$     & 2 & 89579520$^{D_{9}}$, 310542336$^{D_{6}}$     \\
\hline
$L_{16, 1}$  & 1 & 7739670528$^{D_{26}}$                          \\
$L_{16, 3}$  & 2 & 15925248$^{D_{10}}$, 198180864$^{D_{7}}$    \\
\hline
$L_{17}$     & 3 & 2985984$^{D_{12}}$, 80621568, 985374720$^{D_{11}}$ \\
\hline
$L_{18, 1}$  & 1 & 67184640                     \\
$L_{18, 2}$  & 1 & 26873856                     \\
$L_{18, 3}$  & 3 & 1327104$^{D_{15}}$, 17915904, 19906560$^{D_{13}}$  \\
$L_{18, 4}$  & 3 & 1161216$^{D_{16}}$, 2985984$^{D_{14}}$, 6718464    \\
\hline
$L_{19, 1}$  & 1 & 13436928                     \\
$L_{19, 2}$  & 1 & 746496                       \\
$L_{19, 3}$  & 3 & 110592$^{D_{17}}$, 331776$^{D_{18}}$, 497664  \\
\hline
$L_{20, 2}$  & 1 & 11824496640$^{D_{27}}$    \\
$L_{20, 4}$  & 1 & 22394880                     \\
$L_{20, 6}$  & 1 & 1327104                      \\
$L_{20, 7}$  & 2 & 559872, 6718464              \\
$L_{20, 8}$  & 1 & 331776                       \\
$L_{20, 9}$  & 1 & 165888                       \\
$L_{20,10}$  & 1 & 41472                        \\
$L_{20,11}$  & 3 & 184320$^{D_{19}}$, 491520$^{D_{21}}$, 497664     \\
$L_{20,12}$  & 7 & 24576$^{D_{20}}$, 92160$^{D_{22}}$, 124416, 373248,
497664$^{D_{23}}$, 14929920, 99532800$^{D_{24}}$ \\
\hline
$L_{21, 2}$  & 1 & 29113344                     \\
$L_{21, 3}$  & 1 & 82114560                     \\
$L_{21, 4}$  & 1 & 1119744                      \\
$L_{21, 5}$  & 2 & 248832, 3483648              \\
$L_{21, 6}$  & 3 & 41472, 124416, 1866240       \\
$L_{21, 7}$  & 2 & 20736, 1119744               \\
$L_{21, 8}$  & 2 & 13824, 27648                 \\
$L_{21, 9}$  & 3 & 7776, 20736, 31104           \\
$L_{21,10}$  & 4 & 6912, 15552, 41472, 248832   \\
$L_{21,11}$  & 1 & 6912                         \\
$L_{21,12}$  & 2 & 62208, 186624                \\
\noalign{\hrule height0.8pt}
   \end{tabular}
}
\end{center}
\end{table}

\setcounter{table}{2} 
\begin{table}[htbp]
\caption{Ternary self-dual $[24,12,3]$ codes (continued)}
\begin{center}
{\footnotesize
\begin{tabular}{c|r|l}
\noalign{\hrule height0.8pt}
$L$&\multicolumn{1}{c|}{$N$}&\multicolumn{1}{c}{$\#\Aut$} \\
\hline
$L_{22, 8}$  & 1 & 1658880                                         \\
$L_{22,10}$  & 1 & 373248                                          \\
$L_{22,12}$  & 1 & 248832                                          \\
$L_{22,14}$  & 2 & 36864, 110592                                   \\
$L_{22,15}$  & 1 & 96768                                           \\
$L_{22,16}$  & 1 & 27648                                           \\
$L_{22,17}$  & 1 & 20736                                           \\
$L_{22,18}$  & 2 & 6912, 20736                                     \\
$L_{22,19}$  & 2 & 9216, 829440                                    \\
$L_{22,20}$  & 3 & 2304, 6912, 10368                               \\
$L_{22,21}$  & 3 & 1152, 3456, 3456                                \\
$L_{22,22}$  & 2 & 13824, 18432                                    \\
$L_{22,23}$  & 2 & 576,   9216                                     \\
$L_{22,24}$  & 6 & 1152, 1152, 1728, 4608, 6912, 10368             \\
$L_{22,25}$  & 5 & 576, 576, 864, 1152, 10368                      \\
$L_{22,26}$  & 5 & 576, 1728, 2304, 4608, 4608                     \\
$L_{22,27}$  & 2 & 4608, 746496                                    \\
\hline
$L_{23, 4}$  & 1 & 68428800                                        \\
$L_{23,16}$  & 1 & 622080                                          \\
$L_{23,18}$  & 1 & 186624                                          \\
$L_{23,24}$  & 1 & 13824                                           \\
$L_{23,27}$  & 2 & 31104, 311040                                   \\
$L_{23,30}$  & 1 & 3072                                            \\
$L_{23,31}$  & 1 & 3072                                            \\
$L_{23,32}$  & 2 & 2880, 3456                                      \\
$L_{23,33}$  & 2 & 2304, 27648                                     \\
$L_{23,35}$  & 2 & 1152, 3456                                      \\ 
$L_{23,37}$  & 3 & 192, 384, 3456                                  \\
$L_{23,38}$  & 2 & 864, 10368                                      \\
$L_{23,39}$  & 4 & 192, 288, 2592, 15552                           \\
$L_{23,40}$  & 6 & 288, 384, 576, 768, 768, 3456                   \\
$L_{23,41}$  & 2 & 384, 432                                        \\
$L_{23,42}$  & 3 & 96, 144, 192                                    \\
$L_{23,43}$  & 2 & 96, 384                                         \\
$L_{23,44}$  & 9 & 24, 48, 96, 96, 96, 144, 384, 384, 1728         \\
$L_{23,45}$  & 4 & 96, 192, 384, 384                               \\
$L_{23,46}$  & 10& 24, 48, 48, 48, 48, 96, 192, 384, 768, 13824    \\
$L_{23,47}$  & 4 & 60, 72, 96, 288                                 \\
\noalign{\hrule height0.8pt}
   \end{tabular}
}
\end{center}
\end{table}

By Lemma~\ref{Lem:A3A6},
the weight enumerator of a self-dual $[24,12,3]$ code $C$
is determined by the numbers of vectors of norms $1$ and $2$
in the lattice $A_3(C)$.
Since $A_3(C)$ is isomorphic to $L_{i,j}=M_{i,j}\oplus\ZZ^{24-i}$
for some $i,j$, and the kissing numbers of $M_{i,j}$ are given 
in~\cite[Table 16.7]{SPLAG}, we do not give the weight enumerators of
codes in Table~\ref{Table:d3}.

Similar to Section~\ref{Sec:d6},
for each $i=1,2,\ldots,8$ we computed
$T_i$ given in (\ref{Eq:Ti}) and verified that it
coincided with the value determined in~\cite[Table I]{LPS}.
This shows that there is no other self-dual code
with minimum weight $3$.
In addition, the number of distinct self-dual codes of 
length $n$ is known~\cite{MPS} as
\[
N(n)=2 \prod_{i=1}^{(n-2)/2}(3^i+1).
\]
As a check, we verified the mass formula
\[
\sum_{D \in \mathcal{D}_{24}} \frac{2^{24} \cdot 24!}{\#\Aut(D)} =
96722522147893108730806108160000
=N(24),
\]
where $\mathcal{D}_{24}$ denotes the set of all inequivalent
self-dual codes of length $24$.
The mass formula shows that there is no other self-dual code
of length $24$ and the classification is complete.
As a corollary, we have the following:

\begin{cor}
A $24$-dimensional odd unimodular lattice $L$
can be constructed from some ternary self-dual code of length $24$
by Construction~A if and only if
$L$ is isomorphic to one of the lattices given in
Tables~\ref{Table:d6} and~\ref{Table:d3}, and
the odd Leech lattice.
\end{cor}

\section{Even unimodular neighbors of $A_3(C)$}
Let $C$ be a self-dual code of length 
$n \equiv 0 \pmod{12}$ containing the all-one's vector $\1$.
To construct the Niemeier lattices from self-dual codes of length $24$, 
Montague \cite{Montague} considered the following 
constructions of unimodular lattices
\[
L_S(C)=\Big\langle \frac{1}{2\sqrt{3}}\1,B_3(C) \Big\rangle
\text{ and }
L_T(C)=\Big\langle \frac{1}{2\sqrt{3}}\1-e_1,B_3(C) \Big\rangle
\]
which are unimodular neighbors of $A_3(C)$,
where
$B_3(C)=\{v \in A_3(C)|(v,v) \in 2\ZZ\}$ and
$e_1=(\sqrt{3},0,\ldots,0)$.
In particular, if $n \equiv 0 \pmod{24}$, then
$L_S(C)$ and $L_T(C)$ are the even unimodular neighbors of $A_3(C)$.
These constructions 
are called the straight and twisted constructions, respectively
\cite{Montague}.
Montague \cite{Montague} demonstrated that the $23$
Niemeier lattices other than the Leech lattice can be
constructed by the straight construction, and
the $22$ Niemeier lattices other than the two lattices with
root systems $A_{24}$ and $D_{24}$ can be constructed by 
the twisted construction.
He also conjectured that the Niemeier lattice
with root system $D_{24}$ cannot be constructed by
the twisted construction.

We verified that
the code $g_{11}+p_{13}$ given in~\cite[Fig.~5]{LPS}
which is equivalent to the unique code $C$ with
$A_3(C) \cong L_{24,141}$ in Table~\ref{Table:d6},
gives the Niemeier lattice with root system $A_{24}$
by the twisted construction. 
More specifically, let $C$ be the code with 
generator matrix
\begin{equation}\label{eq:genmat}
\left(\begin{array}{cc}
G_{11}           & O_{5 \times 13} \\
O_{6 \times 11}  & P_{13}          \\
00000011111 & 1101000001000
\end{array}\right),
\end{equation}
where $O_{m \times n}$ denotes the $m \times n$ zero matrix,
\[
G_{11}=
\left(\begin{array}{c}
10000201221\\
01000210122\\
00100221012\\
00010222101\\
00001212210
\end{array}\right) \text{ and }
P_{13}=
\left(\begin{array}{c}
1000002212001\\
0100001012202\\
0010002010221\\
0001001022021\\
0000101220201\\
0000011210022
\end{array}\right).
\]
Then the lattice $L_T(C)$ is the Niemeier lattice with root system
$A_{24}$.
Hence the $23$ Niemeier lattices other than the lattice with
root system $D_{24}$ can be constructed by the twisted construction.

In addition, we verified that the lattice with root 
system $D_{24}$ cannot be constructed by the twisted 
construction.
We note, however, that there is a delicate point regarding
the distinction between the two constructions, as described
by the following proposition.

\begin{prop}
Let $C$ be a ternary self-dual code of length $n\equiv0\pmod{12}$.
Suppose that $\1 \in C$ and
$v=(v_1,\ldots,v_{n})$ is a codeword of weight $n$ in $C$.
Let $P$ be the diagonal matrix whose diagonal entries are
the entries of the codeword $v$ regarded as elements of
$\{\pm1\}\subset\ZZ$.
Then 
\[
L_S(C)\cdot P=\begin{cases}
L_S(C\cdot P)&\text{if $\prod_{i=1}^n v_i=1$,}\\
L_T(C\cdot P)&\text{otherwise.}
\end{cases}
\]
\end{prop}
\begin{proof}
Observe $\1\in C\cdot P$. It is easy to see that
$B_3(C)\cdot P=B_3(C\cdot P)$.
Thus
\[
L_S(C)\cdot P
=\Big\langle \frac{1}{2\sqrt{3}}v,B_3(C\cdot P) \Big\rangle,
\]
where $v$ is regarded as a vector of $\{\pm1\}^n\subset\ZZ^n$.
Therefore $L_S(C)\cdot P=L_S(C\cdot P)$ if and only if
\[
\frac{1}{2\sqrt{3}}(v-\1)\in B_3(C\cdot P).
\]
Since $(v-\1,v-\1)\equiv0\pmod{12}$,
this is equivalent to $(v-\1,v-\1)\equiv0\pmod{8}$, or
$\prod_{i=1}^n v_i=1$.
\end{proof}

Therefore, if a self-dual code $C$ contains both $\1$ and
a codeword $v$ of weight $n$ with $\prod_{i=1}^n v_i=-1$, then
$L_S(C)$ (resp.\ $L_T(C)$) is isomorphic to $L_T(C')$ 
(resp.\ $L_S(C')$) for some code $C'$ which is equivalent to $C$.
This means that
the definitions of the straight and the twisted
constructions are independent of the choice of a representative
in the equivalence class of codes only if
$\prod v_i=1$ holds for all codewords $v$ of weight $n$ in $C$.
Self-dual codes satisfying this condition are called admissible
\cite{HKO}. For example, the code $C$ with generator matrix
given in (\ref{eq:genmat}) is not admissible. In fact, there is
a code $C'$ equivalent to $C$ such that $L_T(C')$ is the
Niemeier lattice with root system $A_{12}^2$.

The constructions of Montague \cite{Montague}
have been generalized in \cite{HKO}
for any self-dual codes, not necessarily containing $\1$.
Let $E_4$ be the self-dual $[4,2,3]$ code with generator
matrix
${\displaystyle
\left(\begin{array}{c}
1021\\
0122
\end{array}\right).
}$
The decomposable code $D_1$ is equivalent to the direct sum
$E_4^6$ of six copies of $E_4$, and the Niemeier lattice with
root system $D_{24}$ can be obtained as both $L_S(E_4^6)$ and
$L_T(E_4^6)$, in the notation of \cite{HKO}.
In fact, for any self-dual code $C$ of length $24$,
$\beta_{24}=48 - 21\beta_3 + \beta_6$ holds,
where $\beta_i$ denotes the number of codewords of weight $i$ in $C$
(see \cite[Table III]{LPS}).
Thus, if $C$ has maximum weight less than $24$,
then $C$ has minimum weight $3$.
This implies that $A_3(C)$ has minimum norm $1$, and
hence the two even neighbors of $A_3(C)$ are isomorphic.


\end{document}